\documentclass[12pt,a4paper,twoside,openright]{amsart}

\usepackage{mathtools}
\usepackage{amsmath}
\usepackage{tikz}
\usepackage{stix}
\usepackage{verbatim}
\usepackage{lipsum}
\usepackage[top=1in, bottom=1.25in, left=1.25in, right=1.25in]{geometry}
\usepackage{enumitem}
\usepackage{amsthm}
\usepackage{amssymb}
\usepackage{abstract}
\usepackage{graphicx}
\usepackage{mathtools}
\usepackage{bbm}

\newcommand{\R}{\mathbb{R}}
\newcommand{\Le}{\mathscr{L}}

\newcommand*\circled[1]{\tikz[baseline=(char.base)]{
            \node[shape=circle,draw,inner sep=2pt] (char) {#1};}}

\newtheorem{definition}{Definition}[section]
\newtheorem*{theorem*}{Theorem}
\newtheorem{theorem}{Theorem}[section]

\newtheorem{lemma}{Lemma}[section]
\newtheorem{remark}{Remark}
\newtheorem{proposition}{Proposition}[section]
\numberwithin{equation}{section}
\providecommand{\abstract}{}

\date{}

\title[A $L^p$ stability estimate for regular Lagrangian flows]{A note on a $L^p$ stability estimate for regular Lagrangian flows}
\author{Tommaso Cortopassi}
\thanks{Scuola Normale Superiore, 56126 Pisa, Italy. E-mail: tommaso.cortopassi@sns.it}

\begin{document}

\maketitle

\begin{abstract}
    \footnotesize
    In this note, we prove a $L^p$ version of the well known stability estimate for regular Lagrangian flows derived by Gianluca Crippa and Camillo De Lellis in \cite{crippa2008estimates}. As far as we know, the only estimate of this kind readily available in the literature is for the case $p=1$. With minor modifications to the proof in \cite{crippa2008estimates}, we show that an analogous estimate holds in $L^p$ norm with $p \in (1, + \infty)$. 
\end{abstract}

\section{Introduction}

In order to keep this note short, we won't cover in detail the problem of lack of uniqueness for the ODE

\begin{equation}\label{ODE}
\begin{cases}
    \dot{X} (t)= b(t, X(t,x)) \text{ for } t >0\\
    X(0)= x 
\end{cases}
\end{equation}

in a non classical (i.e. Lipshitz) setting. What we are interested in is the notion of regular Lagrangian flow, which has been proven to provide well-posedness for \eqref{ODE} in the hypotheses

\begin{equation}\label{hypotheses}
b\in L^1 ([0,T]; BV_{loc} (\R^n ; \R^n)) \cap L^\infty ([0,T] \times \R^n ; \R^n) \text{ and } div(b) ^{-} \in L^1 ([0,T] ; L^\infty (\R^n)),
\end{equation}

with $div(b)^-$ denoting the negative part of the divergence. Properties such as existence and uniqueness of the regular Lagrangian flow under the hypotheses \eqref{hypotheses} will be given for granted, the interested reader can look at the seminal paper by Di Perna and Lions \cite{diperna1989ordinary} for the case of $b \in L^1 ([0,T] ; W^{1,p} _{loc} (\R^n; \R^n))$ with $p \geq 1$ and  \cite{ambrosio2004transport} for the generalisation to $BV_{loc}$ vector fields due to Ambrosio. We will use the definition from \cite{ambrosio2004transport} and \cite{ambrosio2014continuity} since it's the one used in \cite{crippa2008estimates}. Notice that this definition is slightly different from the one introduced in \cite{diperna1989ordinary}.

\begin{definition}[\textbf{Regular Lagrangian Flow}] \label{rlf definition}

    Let $b \in L^1 _{loc} ([0,T] \times \R^n; \R^n)$. A map $X: [0,T] \times \R^n \to \R^n$ is a regular Lagrangian flow for $b$ if:

    \begin{enumerate}
        \item for a.e. $x \in \R^n$ the map $x \mapsto X(t,x)$ is an absolutely continuous integral solution of $\dot{\gamma} (t) = b(t, \gamma(t))$ for $ t \in [0,T]$, with $\gamma (0)=x$;

        \item there exists a positive constant $L$, not dependent on $t$, such that 

        $$ \Le^n (X(t, \cdot ) ^{-1} (A)) \leq L \Le^n (A) \text{ for every Borel set } A \subset \R^n, \; \forall t \in [0,T] .$$
  
    Such constant will be called the compressibility constant of $X$.
      \end{enumerate}
\end{definition}

\newpage
\begin{theorem}{\textbf{Existence and uniqueness of the regular Lagrangian flow \cite{ambrosio2004transport}}}

    Let $b$ be a vector field satisfying \eqref{hypotheses}, then there exists a unique regular Lagrangian flow for $b$, in the sense that if $X_i$ are regular Lagrangian flows for $b$ starting from $\Le ^d$-measurable sets $A_i$, with $i=1,2$, then

    $$ B \coloneqq \{x \in A_1 \cap A_2 | X_1 (t,x) = X_2 (t,x) \; \forall t \in [0,T] \}  \text{ has full } \Le ^d \text{ measure in } A_1 \cap A_2 ,$$

    in the sense that $\Le ^d ((A_1 \cap A_2 ) \setminus B)= 0$.
\end{theorem}

The result we want to generalise is the following, which can be found in \cite{crippa2008estimates}:

\begin{proposition}\label{Crippa-De Lellis estimate}{\textbf{\cite{crippa2008estimates}}}

    Let $b$ and $\tilde{b}$ be bounded vector fields belonging to $L^1 ([0,T]; W^{1,p} (\R^n))$ for some $p>1$. Let $X$ and $\tilde{X}$ be regular Lagrangian flows associated to $b$ and $\tilde{b}$ respectively and denote by $L$ and $\tilde{L}$ the compressibility constants of the flow. Then, for every time $\tau \in [0,T]$, we have

    $$ ||X( \tau, \cdot )- \tilde{X} (\tau, \cdot)||_{L^1(B_r (0))} \leq C \left | \log (|| b - \tilde{b} ||_{L^1 ([0, \tau] \times B_R (0))} ) \right|^{-1} ,$$

where $R= r + T|| \tilde{b}||_\infty$ and the constant $C$ only depends on $\tau, r, ||b||_\infty , ||\tilde{b}||_\infty , L, \tilde{L} $ and $||D_x b||_{L^1 (L^p)}$.
\end{proposition}

\begin{remark}
    Actually, in Proposition \ref{Crippa-De Lellis estimate}, the constant $C$ depends on $p$ too. Indeed in the proof of \cite{crippa2008estimates} they use the $L^p$ estimate in Lemma \ref{maximal function estimate}, where the constant $c_{n,p}$ depends also on $p$. 
\end{remark}
\section{Main Result}

The result we will prove is the following:

\begin{proposition}
    Keeping the notations and the assumptions of Proposition \ref{Crippa-De Lellis estimate} it holds that

    \begin{equation}\label{main estimate}
    \sup_{0 \leq t \leq \tau} ||X( t, \cdot )- \tilde{X} (t, \cdot)||_{L^p(B_r (0))} \leq C \left | \log (|| b - \tilde{b} ||_{L^1 ([0, \tau]; L^p ( B_R (0)))} )  \right|^{-1} ,
    \end{equation}

assuming $|| b - \tilde{b} ||_{L^1 ([0, \tau]; L^p ( B_R (0)))}$ to be small enough.
\end{proposition}

\begin{proof}

Let $\delta \coloneqq ||b - \tilde{b} ||_{L^1 ([0, \tau]; L^p (B_{R} (0)))} $ with $R= r + T||\tilde{b} ||_\infty$ and let 

$$ g (t) \coloneqq \left[\int_{B_r (0)} \log \left(1+ \frac{|X (t,x) - \tilde{X} (t,x)|}{\delta} \right)^p dx \right]^{1/p}.$$

For simplicity of notation, we'll omit explicitly writing the dependence on $t$ and $x$ when this won't be necessary. So we'll often write $X $ and $ \tilde{X} $ in lieu of $X(t,x)$ and $ \tilde{X} (t,x)$. Differentiating in time:

\small
\begin{align}\label{g' inequality}
    g' (t) &= \frac{1}{p} g(t) ^{1-p} \int_{B_r (0)} p \log \left( 1 + \frac{|X - \tilde{X}|}{\delta} \right) ^{p-1} \frac{\delta }{\delta + |X - \tilde{X}|}  \frac{1}{\delta} \frac{X - \tilde{X}}{|X - \tilde{X}|} (X - \tilde{X}' )dx \leq \\
    &\leq \frac{g(t)^{1-p}}{\delta}  \int_{B_r (0)} \log \left(  1 + \frac{|X - \tilde{X}|}{\delta} \right) ^{p-1} \frac{1}{1 + |X - \tilde{X}|/\delta} |b (t, X) - \tilde{b} (t, \tilde{X})| dx \leq \nonumber \\
    & \leq \underbrace{\frac{g(t)^{1-p}}{\delta}  \int_{B_r (0)} \log \left(  1 + \frac{|X - \tilde{X}|}{\delta} \right) ^{p-1} \frac{1}{1 + |X - \tilde{X}|/\delta} |b (t, X) - b (t, \tilde{X})| dx}_{\circled{1}} + \nonumber \\
    &+  \underbrace{\frac{g(t)^{1-p}}{\delta}  \int_{B_r (0)} \log \left(  1 + \frac{|X - \tilde{X}|}{\delta} \right) ^{p-1} \frac{1}{1 + |X - \tilde{X}|/\delta} |b (t, \tilde{X}) - \tilde{b} (t, \tilde{X})| dx}_{\circled{2}}. \nonumber 
\end{align}

Regarding $\circled{1}$ we have, thanks to Lemma \ref{pointwise bv inequality}:

\begin{equation}\label{velocity pointwise inequality}
 |b (t, X (t,x)) - b (t, \tilde{X} (t,x))| \leq c_n |X (t,x) - \tilde{X} (t,x)| [M_{\tilde{R}} Db (t, X(t,x)) + M_{\tilde{R}} Db(t, \tilde{X}(t,x))],
 \end{equation}

 for almost every $t \in [0, \tau]$ and (at fixed $t$) for almost every $x \in \R^n$, with $\tilde{R} = T(||b||_\infty + ||\tilde{b}||_\infty)$.

\begin{remark}

Notice that the compressibility constraint in the definition of regular Lagrangian flow is crucial for \eqref{velocity pointwise inequality} to hold almost everywhere for every $t$. Indeed, for almost every $t$ we have that $b_t (x) \coloneqq b (t,x) \in W^{1, p} (\R^n)$. Then \eqref{velocity pointwise inequality} holds only if $X(t,x) , \tilde{X} (t,x) \in \R^n \setminus N_t$, with $N_t$ a null set. Then inequality \eqref{velocity pointwise inequality} holds for almost every $x \in \R^n$ since

$$ \Le ^n (X(t, \cdot ) ^{-1} (N_t)) \leq L \Le ^n (N_t) =0 \text{ and } \Le ^n (\tilde{X} (t, \cdot ) ^{-1} (N_t)) \leq \tilde{L} \Le ^n (N_t) =0.$$
\end{remark}

Using \eqref{velocity pointwise inequality}:

\footnotesize
\begin{align*}
    \circled{1} &\leq c_n g(t) ^{1-p} \int_{B_r (0)} \log \left(  1 + \frac{|X - \tilde{X}|}{\delta} \right) ^{p-1} \frac{|X - \tilde{X}|/\delta}{1 + |X - \tilde{X}|/\delta}[M_{\tilde{R}} Db (t, X(t,x)) + M_{\tilde{R}} Db (t, \tilde{X}(t,x))] dx\leq\\
    &\leq  c_n g(t) ^{1-p} \int_{B_r (0)} \log \left(  1 + \frac{|X - \tilde{X}|}{\delta} \right) ^{p-1}  [M_{\tilde{R}} Db (t, X(t,x)) + M_{\tilde{R}} Db (t, \tilde{X}(t,x))] dx =\\
    &=\underbrace{c_n g(t) ^{1-p} \int_{B_r (0)} \log \left(  1 + \frac{|X - \tilde{X}|}{\delta} \right) ^{p-1} M_{\tilde{R}} Db (t, X(t,x)) dx}_{ \circled{3}} + \\
    &+\underbrace{c_n g(t) ^{1-p} \int_{B_r (0)} \log \left(  1 + \frac{|X - \tilde{X}|}{\delta} \right) ^{p-1} M_{\tilde{R}} Db (t, \tilde{X}(t,x)) dx}_{\circled{4}}.
\end{align*}

\normalsize
For $\circled{3}$, using H\"older inequality with exponents $p$ and $p/(p-1)$ yields

$$ \circled{3} \leq c_n g(t) ^{1-p} g(t) ^{p-1} \left( \int_{B_r (0)} [M_{\tilde{R}} Db (t, X(t,x))]^p dx \right) ^{1/p},$$

and changing variables putting $y=X(t,x)$ (recall condition \textit{(2)} in Definition \ref{rlf definition}) and using Lemma \ref{maximal function estimate} gives

$$ \circled{3} \leq c_n L^{1/p} \left(\int_{B_{r+ T ||b||_\infty} (0)} [M_{\tilde{R}} Db (t,y)]^p dy \right)^{1/p} \leq c_n c_{p,n} L^{1/p} ||Db (t, \cdot)||_{L^p (B_{R'} (0))}, $$

with $R'= r + 3 T \max\{||b||_\infty, ||\tilde{b}||_\infty \} $. With this choice of $R'$, it's easy to see that with the same argument we get

$$ \circled{4} \leq c_n c_{p,n} \tilde{L}^{1/p} ||Db (t, \cdot)||_{L^p (B_{R'} (0))} .$$

As for $\circled{2}$, we clearly have:

\begin{align*}
    \circled{2} \leq \frac{g(t) ^{1-p}}{\delta} \int_{B_r (0)} \log \left(  1 + \frac{|X - \tilde{X}|}{\delta} \right) ^{p-1}|b (t, \tilde{X} (t,x)) - \tilde{b} (t, \tilde{X} (t,x))| dx.
\end{align*}

First using H\"older and then changing variables as before:

\begin{align*}
    \circled{2} \leq \frac{ \tilde{L} ^{1/p} }{\delta } ||b(t, \cdot) - \tilde{b} (t, \cdot) ||_{L^p (B_{R} (0))}    ,
\end{align*}

where we recall that $R= r + T ||\tilde{b}||_\infty$. Inequality \eqref{g' inequality} now reads:

$$ g'(t) \leq c_n c_{p,n} \left( L^{1/p} + \tilde{L} ^{1/p} \right)  ||Db (t, \cdot)||_{L^p (B_{R'} (0))}  + \frac{\tilde{L} ^{1/p}}{\delta} ||b(t, \cdot) - \tilde{b} (t, \cdot) ||_{L^p (B_{R} (0))} .$$

Thanks to the choice of $\delta$, integrating in time from $0$ to $\tau$ we get:

$$ g(t) \leq C \text{ for all } t \in [0, \tau].$$

with $C$ being a constant depending on $ \tau, r, L, \tilde{L}, ||D_x b ||_{L^1 (L^p (B_{R'} (0))} , ||b||_\infty , ||\tilde{b}||_\infty $ and $p$. This means that, for every $t \in [0,\tau]$:

$$ \int_{B_r (0)} \log \left(1+ \frac{|X (t,x) - \tilde{X} (t,x)|}{\delta} \right)^p dx  \leq C^p .$$
 
By Chebychev inequality, for every $\eta>0$ there exists a set $K$ such that $|B_r \setminus K| \leq \eta$ and 
 
 $$ \log \left( \frac{|X (t,x) - \tilde{X} (t,x)|}{\delta} +1\right) ^p \leq \frac{C^p}{\eta} \text{ on }K. $$ 
 
 Then, on $K$:

 $$ |X (t,x) - \tilde{X} (t,x)|^p \leq \delta^p \exp \left(\frac{Cp}{\eta^{1/p}} \right).$$

 Integrating over $B_r$:

\begin{equation} \label{mid Lp estimate}
\int_{B_r (0)} |X (t,x)- \tilde{X} (t,x) |^p dx \leq \eta (||X ||_\infty + ||\tilde{X} ||_\infty ) ^p + \omega_n r^n \delta^p \exp \left(\frac{Cp}{\eta^{1/p}} \right),
\end{equation}

with $\omega _n \coloneqq |B_1 (0)|$. Choosing $\eta= 2^p C^p |\log \delta |^{-p}$ we get

$$\exp \left(\frac{Cp}{\eta^{1/p}} \right)= \exp\left(\frac{p}{2|\log(\delta)|^{-1}}\right)= \exp \left( - \frac{p}{2} \log(\delta) \right)= \exp( \log(\delta^{-p/2})) = \delta ^{-p/2} ,$$

so \eqref{mid Lp estimate} reads as:

\begin{equation}\label{last estimate}   
\int_{B_r (0)} |X(t,x) - \tilde{X} (t,x) |^p 
 dx \leq 2^p C^p (||X ||_\infty + ||\tilde{X} ||_\infty )^p  |\log \delta |^{-p}  + \omega_n r^n \delta ^{p/2} .
 \end{equation} 

Taking the p-th rooth, and assuming that $\delta= || b - \tilde{b} ||_{L^1 ([0, \tau]; L^p ( B_R (0)))}$ is small so that $|\log(\delta)|^{-1} >>\sqrt{\delta}$, we conclude.
\end{proof}

\begin{remark}[Unbounded case]

The theory presented deals only with the bounded case, a generalisation of the theory to unbounded vector fields satisfying additional hypotheses can be found in \cite{ambrosio2014continuity}. In \cite{crippa2008estimates} they also prove an unbounded version of \eqref{main estimate}, whose proof is a modification of the one for bounded vector fields. We believe that, using similar modifications, a $L^p$ estimate should follow as well.
\end{remark}

\section{Appendix}

\begin{definition}[\textbf{Local maximal function}]

    Let $\mu$ be a vector valued locally finite measure. For every $\lambda >0$, we define the local maximal function of $\mu$ as 

    $$ M_\lambda \mu (x) \coloneqq \sup_{0 < r < \lambda} \frac{|\mu| (B_r (x))}{|B_r (x)|} = \sup_{0< r < \lambda} \fint_{B_r (x)} d|\mu| (y) , \qquad \text{for } x \in \R^n.$$

    When $\mu = f \Le ^n$ with $f \in L^1 _{loc} (\R^n ; \R^n)$, we will write $M_\lambda f$ to denote $M_\lambda \mu$.
\end{definition}

The proof of both of the following lemmas can be found in \cite{stein1970singular}.

\begin{lemma}\label{pointwise bv inequality}

    If $ u \in BV (\R^n)$, there exists a negligible set $N \subset \R^n$ such that

    $$ |u(x) - u(y)| \leq c_n |x-y| (M_\lambda Du(x) + M_\lambda Du (y)) \text{ for } x, y \in \R^n \setminus N \text{ with } |x-y| \leq \lambda.$$
\end{lemma}

\begin{lemma}\label{maximal function estimate}
    Let $\lambda>0$ and $p>1$. Then, given $\mu = f \Le ^n$ with $f \in L^1 _{loc} (\R^n ; \R^n)$ and $\rho >0$, it holds that

    $$ ||M_\lambda f ||_{L^p (B_\rho (0))} \leq c_{p,d} ||f||_{L^p (B_{\rho + \lambda} (0))} .$$
\end{lemma}

    \bibliographystyle{plain}
    \bibliography{bibliography.bib}

\end{document}